\documentclass{amsart}
\usepackage[utf8]{inputenc}

\usepackage{amssymb}
\usepackage{amsmath}
\usepackage{amsthm}
\usepackage{graphicx}
\usepackage{tikz}
\usepackage{url}
\usepackage{mathtools}
\usepackage{accents}
\usepackage[backend=biber,style=alphabetic,sorting=nty,doi=false,isbn=false,url=false,eprint=false]{biblatex}
\usepackage{geometry}
\usepackage{colortbl}
\usepackage{float}

\renewbibmacro{in:}{}
\usetikzlibrary{positioning}
\usetikzlibrary{calc}
\usetikzlibrary{decorations.pathreplacing}
\usetikzlibrary{cd}

\newcommand{\scrN}{\mathcal{N}}

\newcommand{\scrI}{\mathcal{I}}

\newcommand{\Q}{\mathbb{Q}}
\newcommand{\R}{\mathbb{R}}

\newcommand{\append}{{}^\frown}
\newcommand{\boldsig}{\boldsymbol{\Sigma}}
\newcommand{\boldpi}{\boldsymbol{\Pi}}
\newcommand{\bolddelta}{\boldsymbol{\Delta}}

\newcommand\forces{\Vdash}

\newcommand{\Pow}{\mathcal{P}}

\newcommand{\GP}{\mathsf{GP}}
\newcommand{\non}{\operatorname{non}}
\newcommand{\cov}{\operatorname{cov}}
\newcommand{\add}{\operatorname{add}}
\newcommand{\cof}{\operatorname{cof}}
\newcommand{\nul}{\mathcal{N}}
\newcommand{\meager}{\mathcal{M}}

\newcommand{\frakb}{\mathfrak{b}}
\newcommand{\frakd}{\mathfrak{d}}
\newcommand{\all}{\mathsf{all}}
\newcommand{\ZFC}{\mathsf{ZFC}}
\newcommand{\ZF}{\mathsf{ZF}}
\newcommand{\AD}{\mathsf{AD}}

\newcommand{\proj}{\operatorname{proj}}

\newcommand{\Det}{\operatorname{Det}}

\newcommand{\I}{\mathcal{I}}

\newcommand{\Haus}{\mathcal{H}}
\newcommand{\diam}{\mathrm{diam}}
\newcommand{\Laver}{\mathbb{L}}

\newcommand{\ctble}{\mathsf{ctble}}

\DeclarePairedDelimiter\abs{\lvert}{\rvert}
\newcommand{\seq}[1]{{\langle#1\rangle}}
\DeclarePairedDelimiterX{\norm}[1]{\lVert}{\rVert}{#1}

\renewcommand\emptyset{\varnothing}
\renewcommand\subset{\subseteq}
\renewcommand{\setminus}{\smallsetminus}

\theoremstyle{definition}
\newtheorem{thm}{Theorem}[section]
\newtheorem*{thm*}{Theorem}
\newtheorem{defi}[thm]{Definition}

\newtheorem*{defi*}{Definition}
\newtheorem{lem}[thm]{Lemma}
\newtheorem*{lem*}{Lemma}
\newtheorem{fact}[thm]{Fact}
\newtheorem*{fact*}{Fact}
\newtheorem{prop}[thm]{Proposition}
\newtheorem*{prop*}{Proposition}

\newtheorem*{rmk*}{Remark}
\newtheorem{cor}[thm]{Corollary}
\newtheorem*{cor*}{Corollary}
\newtheorem{prob}[thm]{Problem}

\title{Goldstern's Principle with respect to Hausdorff Measures}
\author{Tatsuya Goto}
\date{\today}
\address{
	\newline
	Institute of Discrete Mathematics and Geometry, TU Wien \newline
	Wiedner Hauptstrasse 8-10/104, 1040 Wien, Austria
}
\email{goto.tatsuya@icloud.com}

\begin{document}
	\begin{abstract}
		This paper is a continuation of the paper \cite{Goto2025} and studies Goldstern's principle, a principle about unions of continuum many null sets, further.
		The main result is that the Hausdorff measure version of Goldstern's principle for $\boldpi^1_1$ sets fails in $L$, despite the fact that the Lebesgue measure version is true.
		Moreover, we show that this version holds provided that the measurable cardinal exists.
		Other various results regarding Goldstern's principle are established.
	\end{abstract}
	
	\maketitle
	
	
	\section{Introduction}
	
	In \cite{Goldstern1993}, Martin Goldstern showed the following theorem: Let $\seq{A_x : x \in \omega^\omega}$ be a family of Lebesgue measure zero sets. Assume that this family is monotone in the sense that if $x, x' \in \omega^\omega$ satisfy $x \le x'$ (pointwise domination) then $A_x \subset A_{x'}$. Also assume that $A = \{ (x, y) : y \in A_x \}$ is a $\boldsig^1_1$ set. Then $\bigcup_{x \in \omega^\omega} A_x$ also has Lebesgue measure zero.
	
	In  \cite{Goto2025}, the author considered the principle $\GP(\Gamma)$ where $\Gamma$ is a pointclass, replacing $\boldsig^1_1$ by $\Gamma$ in Goldstern's theorem.
	He showed various results in that paper. We summarize the results below without proofs.
	
	\begin{defi}
		Let $\Gamma$ be a pointclass.
		Then $\GP(\Gamma)$ means the following statement:
		Let $A \subset \omega^\omega \times 2^\omega$ be in $\Gamma$.
		Assume that for each $x \in \omega^\omega$, $A_x$ has Lebesgue measure $0$, where $A_x$ is the vertical section of $A$ at $x$.
		Also suppose the monotonicity property: $(\forall x, x' \in \omega^\omega)(x \le x' \Rightarrow A_{x} \subset A_{x'})$.
		Then $\bigcup_{x \in \omega^\omega} A_x$ has also Lebesgue measure $0$.
	\end{defi}
	
	$\all$ denotes the pointclass of all subsets of Polish spaces.
	
	\begin{fact}\label{fact:priorstudies}
		\begin{enumerate}
			\item \cite{Goldstern1993}. $\GP(\boldsig^1_1)$ holds.
			\item Theorem 3.4 of \cite{Goto2025}. $\GP(\boldpi^1_1)$ holds.
			\item Corollary 4.3 and Proposition 4.4 of \cite{Goto2025}. $\GP(\all)$ implies $\add(\nul) \ne \frakb$, $\non(\nul) \ne \frakb$, $\non(\nul) \ne \frakd$ and $\add(\meager) \ne \cof(\meager)$.
			\item Theorem 5.3 of \cite{Goto2025}. The Laver model (the model obtained by countable support iteration of Laver forcing of length $\omega_2$) satisfies $\GP(\all)$.
			\item Theorem 7.1 and Theorem 9.5 of \cite{Goto2025}. $\ZF + \AD \vdash \GP(\all)$ and the Solovay model satisfies $\GP(\all)$, where $\AD$ stands for the axiom of determinacy.
			\item \label{item:regularity} Theorem 2.3 and Theorem 6.3 of \cite{Goto2025}. $\boldsig^1_2$-Lebesgue measurability implies $\GP(\boldsig^1_2)$. Also $\GP(\bolddelta^1_2)$ implies $(\forall a \in \R)(\exists z \in \omega^\omega)(\text{$z$ is a dominating real over $L[a]$})$.
		\end{enumerate}
	\end{fact}

	In this paper we also discuss versions of Goldstern's principle with respect to other ideals than $\scrN$.
	So, letting $\mathcal{I}$ be an ideal on some Polish space $X$, we say $\GP(\Gamma, \mathcal{I})$ holds if for every $A \subset \omega^\omega \times X$ satisfying the monotonicity condition, the condition $A \in \Gamma$ and the condition $A_x \in \mathcal{I}$ for every $x \in \omega^\omega$, we have $\bigcup_{x \in \omega^\omega} A_x \in \mathcal{I}$.
	
	We now outline the structure of the paper.
	In Section \ref{sec:countable}, we treat the countable ideal version of Goldstern's principle.
	In Section \ref{sec:hausdorff}, we study the Hausdorff measure version of Goldstern's principle.
	This section contains the main result in this paper, the Hausdorff measure version of Goldstern's principle for $\boldpi^1_1$ sets fails in $L$ and this version holds if there is a measurable cardinal.
	In Section \ref{sec:miscellaneous}, we look at miscellaneous results on Goldstern's principle for Lebesgue null sets.

	In the rest of this section, we fix our notations and preliminaries.
	
	\begin{defi}
		\begin{enumerate}
			\item We say $F \subset \omega^\omega$ is an unbounded family if $\neg(\exists g \in \omega^\omega)(\forall f \in F)(f \le^* g)$. Put $\frakb = \min \{\abs{F} : F \subset \omega^\omega \text{ unbounded family} \}$.
			\item We say $F \subset \omega^\omega$ is a dominating family if $(\forall g \in \omega^\omega)(\exists f \in F)(g \le^* f)$. Put $\frakd = \min \{\abs{F} : F \subset \omega^\omega \text{ dominating family} \}$.
			\item $\nul$ and $\meager$ denote the Lebesgue measure zero ideal and Baire first category ideal on $2^\omega$, respectively.
			\item $\mathcal{E}$ denotes the $\sigma$-ideal on $2^\omega$ generated by closed Lebesgue null sets.
			\item For an ideal $\I$ on a set $X$: 
			\begin{enumerate}
				\item $\add(\I)$ (the additivity of $\I$) is the smallest cardinality of a family $F$ of sets in $\I$ such that the union of $F$ is not in $\I$.
				\item $\cov(\I)$ (the covering number of $\I$) is the smallest cardinality of a family $F$ of sets in $\I$ such that the union of $F$ is equal to $X$.
				\item $\non(\I)$ (the uniformity of $\I$) is the smallest cardinality of a subset $A$ of $X$ such that $A$ does not belong to $\I$.
				\item $\cof(\I)$ (the cofinality of $\I$) is the smallest cardinality of a family $F$ of sets in $\I$ that satisfies the following condition: for every $A \in \I$, there is $B \in F$ such that $A \subset B$.
			\end{enumerate}
		\end{enumerate}
	\end{defi}
	
	
	\begin{defi}
		A function $f\colon [0, \infty) \to [0, \infty)$ is a {\itshape gauge function} if $f(0) = 0$, and $f$ is right-continuous and nondecreasing.
		
		Let $X$ be a metric space.
		For $A \subseteq X$, $f$ a gauge function and $\delta \in (0, \infty]$, we define
		\[
		\Haus^f_\delta(A) = \inf \{ \sum_{n=0}^\infty f(\diam(C_n)) : C_n \subseteq X \text{ (for $n \in \omega$)} \text{ with } A \subseteq \bigcup_{n \in \omega} C_n \text{ and } (\forall n) (\diam(C_n) \le \delta) \}.
		\]
		For $A \subseteq X$ and $f$ a gauge function, we define
		\[
		\Haus^f(A) = \lim_{\delta \to 0} \Haus^f_\delta(A).
		\]
		We call $\Haus^f(A)$ the {\itshape Hausdorff measure} with gauge function $f$. 
		Let $\scrN^f_X$ denote the set $\{ A \subset X : \Haus^f(A) = 0 \}$.
		
		For a countable open basis $\mathcal{U}$ of $X$, let $\Haus^{\mathcal{U},f}_\delta$ be defined similarly to $\Haus^f_\delta$ assuming each member of covering is in $\mathcal{U}$. Also we adopt the similar definition for $\Haus^{\mathcal{U},f}$.
	\end{defi}

	A gauge function $f$ is called a \emph{doubling} gauge function if there is $r > 0$ such that for every $x > 0$ we have $f(2x) < r f(x)$.
	
	A metric space $X$ is called a \emph{doubling}  space if there is $N > 0$ such that for every $r > 0$ and every open ball $B$ of radius $r$, there are $N$-many open balls of radius $r/2$ covering $B$.
	
	Although $\Haus^f_\delta$ is not outer regular, the following lemma is true.
	
	\begin{lem}\label{lem:conseqofdoublingsp}
		Let $X$ be a doubling metric space and $f$ be a gauge function.
		Then there is a $c > 0$ such that for every $\delta > 0$ and $A \subset X$ we have
		$$
		\Haus^f_\delta(A) \ge c \cdot \inf \{ \Haus^f_\delta(U) : A \subset U \text{ open} \}.
		$$
	\end{lem}
	\begin{proof}
		Take a witness $N > 0$ that $X$ is a doubling metric space.
		Fix $\delta > 0$, $A \subset X$ and $\epsilon > 0$.
		If $\Haus^f_\delta(A)$ is infinite, the conclusion is clear. So we assume $\Haus^f_\delta(A)$ is finite.
		By the definition of $\Haus^f_\delta(A)$, we can take a cover $\seq{C_n : n \in \omega}$ such that $A \subset \bigcup_n C_n$, $\sum_n f(\diam(C_n)) \le \Haus^f_\delta(A) + \epsilon$ and $\diam(C_n) \le \delta$ for every $n$.
		Let $d_n = \diam(C_n)$.
		
		We can assume each $C_n$ is nonempty.
		For each $n$, choose $x_n \in C_n$.
		Since $f$ is right-continuous, for each $n$, we can take $\eta_n > 0$ such that $f(d_n + \eta_n) \le 2 f(d_n)$.
		Then we have $C_n \subset B(x_n, d_n + \eta_n)$.
		We can assume that each $\eta_n$ is smaller than $\delta$.
		
		Put $r_n = d_n + \eta_n$, which is $\le 2 \delta$.
		
		Using doubling property twice, we can cover each $B(x_n, r_n)$ by $N^2$-many balls of radius $r_n / 4$.
		So we can pick a sequence $\seq{x_{n,i} : n \in \omega, i < N^2}$ such that
		$$
		B(x_n, r_n) \subset \bigcup_{i<N^2} B(x_{n,i}, r_n / 4).
		$$
		Each ball has diameter $2 \cdot (r_n / 4) \le \delta$.
		Put $U := \bigcup_{n \in \omega} \bigcup_{i<N^2} B(x_{n,i}, r_n / 4)$.
		$U$ is open and it holds that $A \subset U$. Also we have
		$$
		\Haus^f_\delta(U) \le \sum_{n \in \omega} \sum_{i < N^2} f(r_n / 2) \le 2N^2 \sum_{n \in \omega} f(d_n) \le 2N^2 (\Haus^f_\delta(A) + \epsilon).
		$$
		Since $\epsilon$ was arbitrary, we finished the proof.
	\end{proof}
	
	In addition, as a consequence of doubling gauge functions, we have the following lemma.
	
	\begin{lem}\label{lem:conseqofdoublinggauge}
		Let $X$ be a separable metric space, $D$ be a countable dense subset of $X$, $\mathcal{U} = \{ B(x, r) : x \in D, r \in \Q \}$ and $f$ be a doubling gauge function.
		Then there is $c > 0$ such that for every $A \subset X$ and $\delta > 0$, we have
		$$
		\Haus^{\mathcal{U},f}_{2\delta}(A) \le
		c \cdot \Haus^f_\delta(A).
		$$
	\end{lem}
	\begin{proof}
		Take a witness $r > 0$ that $f$ is a doubling gauge function.
		Fix $\epsilon > 0$.
		By the definition of $\Haus^f_\delta(A)$ , we can take a cover $\seq{C_n : n \in \omega}$ such that $A \subset \bigcup_n C_n$, $\sum_n f(\diam(C_n)) \le \Haus^f_\delta(A) + \epsilon$ and $\diam(C_n) \le \delta$ for every $n$.
		Put $d_n = \diam(C_n)$.
		Since $f$ is right-continuous, for each $n > 0$, we can take $\epsilon_n > 0$ such that $f(d_n + \epsilon_n) \le 2 f(d_n)$.
		For each $n$, choose $x_n \in C_n$.
		Take $y_n \in D$ such that $d(x_n, y_n) < \epsilon_n / 2$.
		We can also assume that each $d_n + \epsilon_n$ is a rational number.
		Therefore $C_n \subset B(y_n, d_n + \epsilon_n)$.
		Also we have $\diam(B(y_n, d_n + \epsilon_n)) \le 2(d_n + \epsilon_n)$.
		
		By the property of doubling, we have
		$$
		f(\diam(B(y_n, d_n + \epsilon_n))) \le f(2(d_n + \epsilon_n)) \le r \cdot f(d_n + \epsilon_n) \le r \cdot 2\cdot f(d_n).
		$$
		So
		$$
		\Haus^{\mathcal{U},f}_{2\delta}(A) \le 2r \cdot (\Haus^f_\delta(A) + \epsilon). \qedhere
		$$
	\end{proof}

	$\mathbb{B}$ and $\mathbb{L}$ denote the random forcing and Laver forcing, respectively.
	
	$\mathcal{SN}$ and $\mathcal{NA}$ denote the strong measure zero ideal and the null additive ideal on $2^\omega$, respectively.
	
	\begin{fact}[\cite{Pawlikowski1996}]
		$\mathcal{SN} = \{ X \subset 2^\omega : \forall E \in \mathcal{E}\ X + E \in \nul \}$.
	\end{fact}
	Recall that the definition of $\mathcal{NA}$ is 
	$$
	\mathcal{NA} = \{ X \subset 2^\omega : \forall N \in \mathcal{N}\ X + N \in \nul \}.
	$$
	So by the above fact, $\mathcal{NA} \subset \mathcal{SN}$ holds.
	
	$\mathsf{IP}$ is the set of interval partitions of $\omega$.
	The standard order $\sqsubset$ on $\mathsf{IP}$ is defined as follows:
	$$
	\seq{I_n : n \in \omega} \sqsubset 
	\seq{J_k : k \in \omega} \iff \forall^\infty k\ \exists n\ I_n \subset J_k.
	$$
	Blass \cite{blass2010combinatorial} showed that $(\mathsf{IP}, \sqsubset)$ and $(\omega^\omega, \le^*)$ are Tukey equivalent.
	
	For a pointclass $\Gamma$, $\Gamma(\mathbb{B})$ stands for Lebesgue measurability for sets in $\Gamma$ and $\Gamma(\mathbb{L})$ stands for Laver measurability for sets in $\Gamma$.
	
	As for the following two facts, the reader may consult \cite{khomskii2011regularity}.
	
	\begin{fact}[Kechris; Spinas; Brendle--Löwe; Ikegami]
		The following are equivalent.
		\begin{enumerate}
			\item $\boldsig^1_2(\mathbb{L})$.
			\item $\bolddelta^1_2(\mathbb{L})$.
			\item For every real $a$, there is a dominating real over $L[a]$.
			\item $\boldsig^1_3$-Laver-absoluteness.
		\end{enumerate}	
	\end{fact}
	
	\begin{fact}
		$\boldsig^1_2(\mathbb{B})$ implies $\boldsig^1_2(\mathbb{L})$.
	\end{fact}
	
	
	The following proposition is a generalization of Theorem 4.2,  Corollary 4.3 and Theorem 5.1 of \cite{Goto2025}, which we omit the proof.
	
	\begin{prop}\label{prop:nulltower}
		Let $\scrI$ be an $\sigma$-ideal on some Polish space $X$.
		\begin{enumerate}
			\item If $\GP(\all, \scrI)$ holds then for every increasing sequence $\seq{A_\alpha : \alpha < \kappa}$ of length $\kappa \in \{ \frakb, \frakd \}$ of members in $\scrI$, we have $\bigcup_{\alpha < \frakb} A_\alpha \in \scrI$.
			\item If $\GP(\all, \scrI)$ holds then $\add(\scrI) \ne \frakb$, $\add(\scrI) \ne \frakd$, $\non(\scrI) \ne \frakb$ and $\non(\scrI) \ne \frakd$ hold.
			\item Assume $\frakb = \frakd$. Then $\GP(\all, \scrI)$ is equivalent to the following statement: for every increasing sequence $\seq{A_\alpha : \alpha < \frakb}$ of members in $\scrI$, we have $\bigcup_{\alpha < \frakb} A_\alpha \in \scrI$.
		\end{enumerate}
	\end{prop}
	
	\section{The countable ideal version of $\GP$}\label{sec:countable}
	
	Before discussing Hausdorff measures, we deal with the countable ideal as a warm-up.
	Let $\ctble$ denote the ideal $[\R]^{\le \aleph_0}$.
	
	\begin{prop}
		$\GP(\boldsig^1_1, \ctble)$ holds.
	\end{prop}
	\begin{proof}
		Let $A \subset \omega^\omega \times \R$ be a monotone $\boldsig^1_1$ set with $A_x \in \ctble$ for every $x \in \omega^\omega$.
		Suppose that $\bigcup_{x \in \omega^\omega} A_x$ is uncountable.
		Since this union is $\boldsig^1_1$ set, we can take a perfect set $P \subset \R$ that is contained in $\bigcup_{x \in \omega^\omega} A_x$.
		Let $s$ be a Sacks real over $V$ contained in $P$.
		Since the Sacks real avoids countable sets in the ground model, we have $s \not \in A_x$ for every $x \in \omega^\omega \cap V$.
		But the Sacks forcing is $\omega^\omega$-bounding and we assumed $A$ is monotone, we have $s \not \in A_x$ for every $x \in \omega^\omega$.
		It is a contradiction.
	\end{proof}
	
	\begin{prop}
		$\GP(\all, \ctble)$ is equivalent to $\frakb > \aleph_1$.
	\end{prop}
	\begin{proof}
		That $\GP(\all, \ctble)$ implies $\frakb > \aleph_1$ is easy using Proposition \ref{prop:nulltower} since $\add(\ctble) = \aleph_1$.
		To prove the converse, we assume $\frakb > \aleph_1$.
		Let $\seq{A_x : x \in \omega^\omega}$ be a monotone family of countable sets and assume $\bigcup_{x \in \omega^\omega} A_x$ be uncountable.
		Take a set $B$ of size $\aleph_1$ such that $B \subset \bigcup_{x \in \omega^\omega} A_x$.
		For each $y \in B$, take $x_y$ such that $y \in A_{x_y}$.
		Take a real $x^*$ dominating all $x_y$ for $y \in B$.
		Then we have $Y \subset A_{x^*}$, which is a contradiction.
	\end{proof}
	
	\begin{prop}\label{prop:pi11scale}
		$V=L$ implies there is an uncountable $\boldpi^1_1$ set $S \subset \omega^\omega$ which forms a scale.
	\end{prop}
	\begin{proof}
		This proposition is an application of Miller's coding argument (see \cite{miller1989infinite}).
		
		For $\alpha < \omega_1$, we say $(L_\alpha, \in)$ is point-definable if $\mathsf{SkHull}^{(L_\alpha, \in)}(\emptyset)$ is isomorphic to $(L_\alpha, \in)$.
		Here, $\mathsf{SkHull}^{(L_\alpha, \in)}(\emptyset)$ is the Skolem hull of $\emptyset$ in the structure $(L_\alpha, \in)$ using $L$-least interpretations of Skolem terms. It is known that there exist unbounded many $\alpha < \omega_1$ such that $(L_\alpha, \in)$ is point-definable.
		It is also known that if $(L_\alpha, \in)$ is point-definable, then there is $E \in L_{\alpha + 3}$ such that ($\omega, E) \simeq (L_\alpha, \in)$.
		Let $\mathsf{Th} = \ZFC^* + (V=L)$, where $\ZFC^*$ is a sufficient large finite portion of $\ZFC$.
		
		Let $\seq{y_\alpha : \alpha < \omega_1}$ be the enumeration of reals ordered by the canonical ordering of $L$.
		Recursively construct $\seq{\xi_\alpha, E_\alpha, x_\alpha : \alpha < \omega_1}$ so that
		\begin{enumerate}
			\item $\xi_\alpha$ is the minimum ordinal such that $\xi_\beta < \xi_\alpha$ (for $\beta < \alpha$) and $y_\alpha, \seq{\xi_\beta, E_\beta, x_\beta : \beta < \alpha} \in L_{\xi_\alpha}$ and $L_{\xi_\alpha}$ is point-definable.
			\item $E_\alpha$ is the $L$-least real such that $(\omega, E_\alpha) \simeq (L_{\xi_\alpha}, \in)$.
			\item $x_\alpha$ is the $L$-least real such that $E_\alpha \le_T x_\alpha$, $x_\beta <^* x_\alpha$ for every $\beta < \alpha$, $y_\alpha <^* x_\alpha$.
		\end{enumerate}
		It can be verified that $x_\alpha \in L_{\xi_\alpha + \omega}$.
		
		Then $S = \{ x_\alpha : \alpha < \omega_1 \}$ suffices.
		That $S$ is $\boldpi^1_1$ can be checked by the following equivalence.
		\begin{align*}
			x \in S \iff &(\exists E\ (\Delta^1_1 \text{ in } x)) \\
			&[E\text{ is well-founded reltation on $\omega$,} \\
			&\text{the Mostowski collapse $M$ of $(\omega, E)$ models $\mathsf{Th}$,} \\
			&\text{$M$ has a sequence $v$ of length of a countable successor ordinal,} \\
			&\text{$v$ is the unique sequence of $x_\alpha$ satisfying the above items, and} \\
			&\text{the last element of $v$ is $x$}].
		\end{align*}
		Here, note that the real coding the ordinal $\xi_\alpha$ is computable from $x_\alpha$, $\xi_\alpha + \omega \in \omega_1^{x_\alpha}$ and the real $E$ such that $(\omega, E) \simeq (L_{\xi_\alpha + \omega}, \in)$ is $\Delta^1_1$ in  $x$. 
	\end{proof}
	
	\begin{thm}\label{thm:gppi11ctble}
		$V=L$ implies $\neg \GP(\boldpi^1_1, \ctble)$.
	\end{thm}
	\begin{proof}
		Let $S$ be an uncountable $\boldpi^1_1$ scale.
		Let $A_x = \{ y \in S : y \le^* x \}$.
		Then the set $A$ is $\boldpi^1_1$.
		Since $S$ is scale of length $\omega_1$, each $A_x$ is countable. Also $\bigcup_{x \in \omega^\omega} A_x = S$, which is uncountable.
	\end{proof}
	
	\section{The Hausdorff measure version of $\GP$}\label{sec:hausdorff}
	
	In this section, we investigate the Hausdorff measure version of $\GP$. 
	
	For a gauge function $f$, let $I_{\sigma f}$ be the ideal of all sets of reals of $\sigma$-finite $f$-Hausdorff measure.
	Recall that $\mathbb{P}_{\mathcal{I}}$ stands for the idealized forcing for an ideal $\mathcal{I}$, that is the Borel algebra modulo the ideal $I$.
	
	\begin{thm}
		Let $f$ be a continuous doubling gauge function and $X$ be a compact metric space.
		Then $\GP(\boldsig^1_1, \scrN^f_X)$ holds.
	\end{thm}
	\begin{proof}
		Let $A \subset \omega^\omega \times X$ be a $\boldsig^1_1$ set satisfying the monotonicity condition and assume $A_x$ is of $f$-Hausdorff measure $0$ for each $x \in \omega^\omega$.
		Assume also that $\bigcup_{x \in \omega^\omega} A_x$ is not of $f$-Hausdorff measure $0$.
		
		We divide the argument into two cases.
		
		First, we consider the case where $\bigcup_{x \in \omega^\omega} A_x$ is of non-$\sigma$-finite $f$-Hausdorff measure.
		Since $\bigcup_{x \in \omega^\omega} A_x$ is $\boldsig^1_1$, we can take a Borel set $B$ such that $B \subset \bigcup_{x \in \omega^\omega} A_x$ and $B$ is of non-$\sigma$-finite $f$-Hausdorff measure.
		The existence of such a Borel set $B$ is ensured by the assumption $f$ is continuous and $X$ is compact and the theorem of M. Sion and D. Sjerve \cite[Theorem 6.6]{Sion_Sjerve_1962}.
		
		Let $G$ be a $(V, \mathbb{P}_{I_{\sigma f}})$-generic filter such that $B \in G$. And let $g$ be the generic real corresponding to $G$.
		By genericity, we have $g \not \in A_x$ for each $x \in \omega^\omega \cap V$. Thus we have $g \not \in \bigcup_{x \in \omega^\omega \cap V} A_x$.
		Since the monotonicity condition holds, which is absolute between $V$ and $V[G]$, and since $\mathbb{P}_{I_{\sigma f}}$ is $\omega^\omega$-bounding, which is the Theorem by J. Zapletal \cite[Corollary 4.4.2]{zapletal2008forcing}, we have $g \not \in \bigcup_{x \in \omega^\omega} A_x$. It contradicts the fact that $g \in B \subset \bigcup_{x \in \omega^\omega} A_x$.
		
		Second, we consider the case where $\bigcup_{x \in \omega^\omega} A_x$ is of $\sigma$-finite $f$-Hausdorff measure.
		Let $\seq{B_n : n \in \omega}$ be a sequence of Borel sets such that $\bigcup_{x \in \omega^\omega} A_x \subset \bigcup_n B_n$ and each $B_n$ is of finite positive $f$-Hausdorff measure.
		For each $n$, consider the set $\bigcup_{x \in \omega^\omega} (A_x \cap B_n)$.
		Since $B_n$ is of finite positive $f$-Hausdorff measure, the restriction of $f$-Hausdorff measure to the Borel subsets in $B_n$ is measure isomorphic to the Lebesgue measure by using measure isomorphism theorem.
		Therefore, using $\GP(\boldsig^1_1)$, we conclude $\bigcup_{x \in \omega^\omega} (A_x \cap B_n)$ is of $f$-Hausdorff measure zero.
		So taking the union over $n \in \omega$, we deduce that $\bigcup_{x \in \omega^\omega} A_x$ is of $f$-Hausdorff measure zero.
	\end{proof}
	
	In the following, we employ the notions of the effective Hausdorff dimension and Martin-Löf randomness. See \cite{slaman2021capacitability} and \cite{lutz2018algorithmic}.
	
	\begin{fact}[Proposition 6 of \cite{slaman2021capacitability}]\label{fact:slaman}
		There is an infinite co-infinite recursive set $R \subset \omega$ such that for every reals $B, X, X^*$, if $X$ is Martin-Löf random relative to $B$ and $X \upharpoonright (\omega \setminus R) = X^* \upharpoonright (\omega \setminus R)$, then the effective dimension of $X^*$ relative to $B$ is equal to $1$.
	\end{fact}
	
	\begin{fact}[\cite{lutz2018algorithmic}]\label{fact:lutz}
		Let $A \subset 2^\omega$.
		Then the Hausdorff dimension of $A$ is equal to the infimum over $b \in \R$ of the supremum, over points $x \in A$, of the effective dimension of $x$ relative to $b$.
	\end{fact}
	
	\begin{thm}\label{thm:gppi11haus}
		Suppose $V=L$.
		Then, there is a monotone $\boldpi^1_1$ set $A \subset \omega^\omega \times 2^\omega$ whose vertical sections are countable and $\bigcup_{x \in \omega^\omega} A_x$ is of Hausdorff dimension $1$.
		In particular, in $L$, $\GP(\boldpi^1_1, \scrN^f_{2^\omega})$ fails for every $f = \mathsf{pow}_s$ for $0 < s < 1$. Here $\mathsf{pow}_s(x) = x^s$.
	\end{thm}
	\begin{proof}
		This theorem refines Theorem \ref{thm:gppi11ctble}.
		Let $R = \{ l_0 < l_1 < \dots \}$ be the real from Fact \ref{fact:slaman}.
		Let $\pi \colon \omega^2 \to \omega$ be a recursive bijection.
		For $r \in 2^\omega$, we define $\mathsf{decode}(r) \in \omega^\omega$ by $\mathsf{decode}(r)(n) = m \iff r(l_{\pi(n, m)}) = 1$. 
		Let $\seq{y_\alpha : \alpha < \omega_1}$ be the enumeration of reals ordered by the canonical ordering of $L$.
		Recursively construct $\seq{\xi_\alpha, E_\alpha, x_\alpha, r_\alpha, r^*_\alpha : \alpha < \omega_1}$ so that
		\begin{enumerate}
			\item $\xi_\alpha$ is the minimum ordinal such that $\xi_\beta < \xi_\alpha$ (for $\beta < \alpha$) and $y_\alpha, \seq{\xi_\beta, E_\beta, x_\beta : \beta < \alpha} \in L_{\xi_\alpha}$ and $L_{\xi_\alpha}$ is point-definable.
			\item $E_\alpha$ is the $L$-least real such that $(\omega, E_\alpha) \simeq (L_{\xi_\alpha}, \in)$.
			\item $x_\alpha$ is the $L$-least real such that $E_\alpha \le_T x_\alpha$, $x_\beta <^* x_\alpha$ for every $\beta < \alpha$, $y_\alpha <^* x_\alpha$.
			\item $r_\alpha$ is the $L$-least Martin-Löf random relative to $y_\alpha$.
			\item $r^*_\alpha$ is the $L$-least real such that $r_\alpha \upharpoonright (\omega \setminus R) = r^*_\alpha \upharpoonright (\omega \setminus R)$ and $\mathsf{decode}(r^*_\alpha) = x_\alpha$.
		\end{enumerate}
		It can be verified that $r^*_\alpha \in L_{\xi_\alpha + \omega}$.
		Then $S = \{ r^*_\alpha : \alpha < \omega_1 \}$ is $\boldpi^1_1$ by the same reason as Proposition \ref{prop:pi11scale}.
		Also $S$ is of Hausdorff dimension $1$ by Lutz--Lutz theorem (Fact \ref{fact:lutz}).
		
		Put $A = \{ (x, y) : y \in S, \mathsf{decode}(y) \le^* x \}$.
		Then $\bigcup_x A_x = S$, which is of Hausdorff dimension $1$.
		$A$ is obviously monotone.
		Also each vertical section $A_x$ is countable.
	\end{proof}
	
	\begin{defi}
		Let $X$ be a Polish space.
		A function $c \colon \Pow(X) \to \R_{\ge 0} \cup \{ \infty\}$ is called \emph{capacity} if the following conditions hold:
		\begin{enumerate}
			\item $c(\emptyset) = 0$,
			\item $A \subset B \subset X$ implies $c(A) \le c(B)$,
			\item $c$ is continuous in countable increasing unions, and
			\item For a compact set $K \subset X$, we have $c(K) < \infty$ and $c(K) = \inf \{ c(O) : K \subset O \text{ open} \}$.
		\end{enumerate}
		A capacity $c$ is outer regular, if for every $A \subset X$, we have
		$$c(A) = \inf \{ c(O) : A \subset O \text{ open} \}.$$
	\end{defi}
	
	\begin{defi}
		Let $X$ be a Polish space.
		A function $\mu \colon \Pow(X) \to \R_{\ge 0}$ is called \emph{pavement submeasure} if there is $U \subset \Pow(X)$ and $w \colon U \to \R_{\ge 0}$ such that 
		$$\mu(A) = \inf \{ \sum_{u \in V} w(u) : V \subset U \text{ and } A \subset \bigcup V \}$$
		holds for every $A \subset X$.
	\end{defi}
	
	For a function $\mu \colon \Pow(X) \to \R_{\ge 0} \cup \{ \infty \}$, we say a set $A \subset X$ is \emph{capacitable} for $\mu$ if $\mu(A) = \sup \{ \mu(K) : K \subset A \text{ compact} \}$.
	Also we say $A$ is is \emph{capacitable modulo a constant} $c > 0$ for $\mu$ if $\mu(A) \le c \cdot \sup \{ \mu(K) : K \subset A \text{ compact} \}$.
	
	\begin{fact}[Theorem 3.6.11 of \cite{zapletal2008forcing}; Proposition 4.1 of \cite{zapletal2005preservationtheorems}]\label{fact:laverpres}
		Let $\phi$ be an outer regular capacity on a Polish space $X$ and suppose $\boldpi^1_1$ sets are capacitable for $\phi$.
		Then, the Laver forcing preserves $\phi$.
	\end{fact}
	
	Note that $\Haus^s_\delta$ is not outer regular. Although $\Haus^s$ is outer regular, it is not a capacity because compact sets can have the infinite measure.
	But the similar proof shows the following.
	
	\begin{cor}\label{cor:laverpres}
		Suppose $X$ is a doubling Polish space, $f$ is a doubling gauge function, $D$ be a countable dense subset of $X$ and $\mathcal{U} = \{ B(x, r) : x \in D, r \in \Q \}$.
		Assume $\boldpi^1_1$ sets are capacitable for $\Haus^{\mathcal{U},f}_\delta$ modulo a constant not depending $\delta$.
		Then:
		\begin{enumerate}
			\item There is a constant $c > 0$ such that if $A \subset X$ and $\delta > 0$ then $\Laver \forces \Haus^{\mathcal{U},f}_\delta(\check{A}) \ge c \cdot ({\Haus^{\mathcal{U},f}_{2\delta}(A)})^{\check{}}$.
			\item There is a constant $c > 0$ such that if $A \subset X$ then $\Laver \forces \Haus^{\mathcal{U},f}(\check{A}) \ge c \cdot ({\Haus^{\mathcal{U},f}(A)})^{\check{}}$.
		\end{enumerate}
	\end{cor}
	\begin{proof}
		The proof for (1) is the same as that for Fact \ref{fact:laverpres}.
		Here, we use the fact that there is a constant $d > 0$ such that $\Haus^{\mathcal{U},f}_\delta(A) \ge d \cdot \inf \{ \Haus^{\mathcal{U},f}_{2\delta}(O) : O \text{ is open and } A \subset O \}$, which is a consequence of being a doubling space and being doubling gauge function, using Lemma \ref{lem:conseqofdoublingsp} and Lemma \ref{lem:conseqofdoublinggauge}.
		
		For the sake of completeness, we include the proof, which is almost identical to the proof for Fact \ref{fact:laverpres}.
		Assume that $\boldpi^1_1$ sets are capacitable for $\Haus^{\mathcal{U},f}_\delta$ modulo a constant $c$ which does not depend on $\delta$.
		
		Let $T \in \Laver$, and let $\dot O$ be a name for an open set such that
		\[
		T \forces \Haus^{\mathcal{U},f}_{2\delta}(\dot{O}) \leq \epsilon .
		\]
		What we must show is that there is a Laver tree $S\le T$ such that
		\[
		\Haus^{\mathcal{U},f}_{2\delta}\bigl(\{x\in X : S\forces \check x\in \dot O\}\bigr)
		\le \epsilon \cdot \text{(a constant)} .
		\]
		
		Enumerate $\mathcal{U}$ as $\seq{ V_n : n\in\omega}$.
		Using a fusion argument, find $S \in \Laver$ such that for every node $s\in S$,
		the condition $S_s$ decides, for all $n<|s|$, the statement
		$\dot V_n\subset \dot O$.
		To simplify notation, assume $S=\omega^{<\omega}$.
		
		Consider the space $Y=\omega^\omega\times X$, and define an operator
		$\Gamma:\Pow(Y)\to\Pow(Y)$ by
		\[
		\Gamma(B) = \{ 
		(s,x) : 
		(s,x) \in B\ \lor\ \forall^\infty n\ (s \append n, x) \in B \}.
		\]
		This is a monotone, inductive coanalytic operator. Hence, by Theorem 1.6 of
		Cenzer and Mauldin \cite{cenzer1980inductive}, given a coanalytic set $A\subset Y$, the sequence
		\[
		A=A_0,\quad A_{\alpha+1}=\Gamma(A_\alpha),\quad
		A_\alpha=\bigcup_{\beta<\alpha} A_\beta\ \ (\alpha\ \text{a limit ordinal})
		\]
		stabilizes at $\omega_1$ and reaches a coanalytic set $A_{\omega_1}$.
		Moreover, for every analytic set $C\subset A_{\omega_1}$, there is an ordinal
		$\alpha<\omega_1$ such that $C\subset A_\alpha$.
		
		Let
		\[
		A=\{(s,x)\in Y : \exists n<\abs{s}\ 
		S_s\forces \dot V_n\subset \dot O \text{ and } x\in V_n\}.
		\]
		Write $A^s=\{x\in X : (s, x) \in A\}$. Then
		$s\subset t \Rightarrow A^s\subset A^t$, and these sets satisfy
		$\Haus^{\mathcal{U},f}_{2\delta}(A^s)\le \epsilon$.
		This property is preserved when we pass to $A^s_\alpha$.
		Indeed, to show $\Haus^{\mathcal{U},f}_{2\delta}(A^s_{\alpha+1})\le\epsilon$, note that
		$A^s_{\alpha+1}$ is the increasing union of the family
		$
		\bigcap_{m>n} A^{s\append m}_\alpha \  (n\in\omega),
		$
		and each member of this family has $\Haus^{\mathcal{U},f}_{2\delta}$-capacity $\le \epsilon$.
		So we can use the continuity of $\Haus^{\mathcal{U},f}_{2\delta}$ for increasing unions.
		At limit stages, the same continuity gives
		$\Haus^{\mathcal{U},f}_{2\delta}(A^s_\alpha)\le \epsilon$.
		
		By transfinite induction we can show that
		\[
		A^\emptyset_{\alpha_1}=\{x\in X : S\forces \check x\in \dot O\}.
		\]
		
		Therefore it is enough to show $\Haus^{\mathcal{U},f}_{2\delta}(A^\emptyset_{\alpha_1})\le \epsilon/c$.
		But if $\Haus^{\mathcal{U},f}_{2\delta}(A^\emptyset_{\alpha_1})>\epsilon/c$, then by capacitability of $A^\emptyset_{\alpha_1}$
		there is a compact set $C\subset A^\emptyset_{\alpha_1}$ with $\Haus^{\mathcal{U},f}_{2\delta}(C)>\epsilon$.
		Such a set must be contained in $A^\emptyset_{\alpha}$ for some countable ordinal $\alpha$.
		But we have $\Haus^{\mathcal{U},f}_{2\delta}(A^\emptyset_{\alpha})\le\epsilon$, a contradiction.
		
		(2) follows from (1).
	\end{proof}
	
	Let $\Det(\boldsig^1_1)$ denote the axiom of determinacy for $\boldsig^1_1$ sets.
	
	\begin{fact}[Corollary 5.2 of \cite{zapletal2005preservationtheorems}]\label{fact:regularitypavement}
		Let $\mu$ be a pavement submeasure on some Polish space $X$, derived from countable set $U$ of Borel pavers with weight function $w$.
		If $\Det(\boldsig^1_1)$ holds then every $\boldpi^1_1$ set has an analytic subset of the same submeasure.
	\end{fact}
	
	Note that $\Haus^{\mathcal{U},f}_\delta$ is a pavement submeasure when $\mathcal{U}$ is a countable open basis.
	So we can apply Fact \ref{fact:regularitypavement} to $\Haus^{\mathcal{U},f}_\delta$ and can have the same conclusion for $\Haus^{\mathcal{U},f}$ taking union for $\delta = 2^{-n}$ for all $n$.
	
	Using this observation and Howroyd's theorem \cite{howroyd1995dimension} stating that every analytic set are capacitable for $\Haus^{f}$, we can deduce that if $\Det(\boldsig^1_1)$ holds then every $\boldpi^1_1$ set is capacitable for $\Haus^{\mathcal{U},f}$ modulo a constant.
	
	\begin{thm}\label{thm:consgppi11haus}
		If $\boldpi^1_1$ sets are capacitable for $\Haus^{\mathcal{U},f}$ in both of $V$ and the one-step Laver extension $V^\Laver$, then $\GP(\boldpi^1_1, \scrN^f_{X})$ holds for every doubling Polish space $X$ and a doubling gauge function $f$.
	\end{thm}
	\begin{proof}
		Fix a countable basis $\mathcal{U}$ of $X$.
		Let $A \subset \omega^\omega \times X$ be monotone $\boldpi^1_1$ set with vertical sections in $\scrN^f_X$.
		Let $\phi(x)$ be the formula defined by
		$$
		\phi(x) \equiv \neg (\exists K \subset X \text{ compact})(K \subset A_x \text{ and } \Haus^{\mathcal{U},f}(K) > 0).
		$$
		The part $\Haus^{\mathcal{U},f}(K) > 0$ is arithmetic because we can consider only finite cover by basic open sets using compactness.
		So $\phi(x)$ is $\boldpi^1_2$.
		Also $A_x \in \scrN^f_X$ implies $\phi(x)$ and converse holds when if $\boldpi^1_1$ sets are capacitable.
		
		Work in $V[d]$, where $d$ is a Laver real over $V$.
		The statement  $(\forall x)\phi(x)$ is $\boldpi^1_2$ and absolute between $V$ and $V[d]$.
		So using the above remark, we have $A_d \in \scrN^f_X$.
		By using monotonicity, the property of dominating real and absoluteness, we have $(\bigcup_x A_x)^V \subset A_d$.
		Since Laver forcing preserves $\Haus^f$-positivity under our assumption using Corollary \ref{cor:laverpres}, we have $\bigcup_x A_x \in \scrN^f_X$ in $V$.
	\end{proof}
	
	\begin{cor}\label{cor:gppi11haus}
		If there is a measurable cardinal, then $\GP(\boldpi^1_1, \scrN^f_{X})$ holds for every doubling Polish space $X$ and a doubling gauge function $f$.
	\end{cor}
	\begin{proof}
		This Corollary follows from Fact \ref{fact:regularitypavement} and \ref{thm:consgppi11haus}.
		Note that if there is a measurable cardinal $\kappa$, then small forcing (like the Laver forcing) forces also $\kappa$ is a measurable cardinal and that existence of a measurable cardinal implies $\Det(\boldsig^1_1)$.
	\end{proof}

	\section{Miscellaneous results on $\GP$ for Lebesgue measure}\label{sec:miscellaneous}
	
	\subsection{$\GP$ for pointclasses between $\boldpi^1_1$ and $\bolddelta^1_2$}\label{sec:suslinoperator}
	
	We know that $\GP(\boldpi^1_1)$ is true, but $\GP(\bolddelta^1_2)$ is not necessary true. Therefore, it is worthwhile analyzing $\GP$ for pointclasses between $\boldpi^1_1$ and $\bolddelta^1_2$.
	
	Recall that for $P \subset X \times \omega^{<\omega}$, $\mathcal{A}(P) = \bigcup_{z \in \omega^\omega} \bigcap_{n \in \omega} A^{z \upharpoonright n}$. This $\mathcal{A}$ is called the Suslin operator. For a pointclass $\Gamma$, we let $\mathcal{A}(\Gamma) = \{ \mathcal{A}(P) : P \in \Gamma \}$.
	Also for a pointclass $\Gamma$, let ${\sim}\Gamma = \{ A^\mathrm{c} : A \in \Gamma \}$.
	
	\begin{thm}
		$\GP(\mathcal{A}(\boldpi^1_1))$ holds. 
	\end{thm}
	\begin{proof}
		Note that every set in $\mathcal{A}(\boldpi^1_1)$ is Lebesgue measurable.
		We use the same method as the proof of $\GP(\boldpi^1_1)$, i.e. the proof using Laver forcing and absoluteness.
		The only nontrivial part is that the statement $``\mu(A_x) = 0"$ is $\boldpi^1_2$ for $A \in \mathcal{A}(\boldpi^1_1)$.
		Let $A = \mathcal{A}(P)$, where $P$ is $\boldpi^1_1$.
		Now we consider the following statement.
		$$
		\varphi(x) :\iff (\exists B \in  2^\omega \times \omega^{<\omega} \text{ Borel})[\mu(\mathcal{A}(B)) > 0 \text{ and } (\forall s \in \omega^{<\omega}) B^s \subset (P^s)_x ].
		$$
		The complexity of this statement is $\boldsig^1_2$.
		Here we used $\mathcal{A}(B)$ is $\boldsig^1_1$ sets and by Kechris--Tanaka theorem ({{\cite{kechris1973337, tanaka196711}}}), the statement $\mu(\mathcal{A}(B)) > 0$ is $\boldsig^1_1$.
		
		It is sufficient to prove $\mu(A_x) > 0 \iff \varphi(x)$.
		The $\Leftarrow$ part is clear.
		To prove the $\Rightarrow$ part, take a Borel set $B^s \subset (P^s)_x$ with the same measure as $(P^s)_x$ for each $s \in \omega^{<\omega}$.
		Put $B = \{ (y, s) : y \in B^s \}$.
		Since we have $\mathcal{A}(P_x) \setminus \mathcal{A}(B) \subset \bigcup_{s \in \omega^\omega} (P_x^s \setminus B^s)$, it holds that $\mu( \mathcal{A}(B)) = \mu(\mathcal{A}(P_x)) > 0$.
	\end{proof}
	
	\begin{fact}[Theorem 29.27 of \cite{kechris2012classical}]\label{fact:measurecpt}
		Let $Z, W$ be Polish spaces and $F \subseteq Z \times W$ be closed. If $\mu$ is a Borel probability measure on $Z$ and for some $a \in \mathbb{R}, \mu\left(\operatorname{proj}_Z(H)\right)>a$, then there is a compact set $K \subseteq F$ such that $\mu\left(\operatorname{proj}_Z(K)\right)>a$.
	\end{fact}
	
	\begin{thm}
		$\GP({\sim}\mathcal{A}(\boldpi^1_1))$ holds. 
	\end{thm}
	\begin{proof}
		We have to check the statement $``\mu(A_x) = 0"$ is $\boldpi^1_2$ for $A \in {\sim} \mathcal{A}(\boldpi^1_1)$.
		
		Let $A \in {\sim} \mathcal{A}(\boldpi^1_1)$. Then we can write
		$A = \bigcap_{y \in \omega^\omega} \bigcup_{n \in \omega} P^{y \restriction n}$, where $P$ is $\boldsig^1_1$.
		Write $P = \proj(F)$, where $F \subset \omega^\omega \times \omega^\omega \times 2^\omega \times \omega^{<\omega}$ is a closed set.
		We consider the following statement:
		\begin{align*}
			&(\exists \seq{K^m_s : m \in \omega, s \in \omega^{<\omega}}) \\ &\hspace{5mm}[\text{each $K^m_s$ is a compact subset of $\mathsf{section}^1_x(F^s)$ and } \\
			&\hspace{10mm}\mu\left( \bigcap_{y \in \omega^\omega} \bigcup_{n \in \omega} \bigcup_{m \in \omega} \proj(K^m_{y \restriction n})\right) > 0 ]
		\end{align*}
		Here $\mathsf{section}^1_x(F^s)$ means $\{ (z, y) : (z, x, y, s) \in F \}$.
		
		Since each $K^m_s$  is compact, $\proj(K^m_{y \restriction n})$ is also compact. Thus, $\bigcap_{y \in \omega^\omega} \bigcup_{n \in \omega} \bigcup_{m \in \omega} \proj(K^m_{y \restriction n})$ is a $\boldpi^1_1$ set. So, $\mu\left( \bigcap_{y \in \omega^\omega} \bigcup_{n \in \omega} \bigcup_{m \in \omega} \proj(K^m_{y \restriction n})\right) > 0$ is a $\boldpi^1_1$ statement using Kechris--Tanaka theorem.
		So the entire statement is $\boldsig^1_2$.
		
		We now claim that $\mu(A_x) > 0$ if and only if the above statement holds.
		``If" part is obvious. For ``only if" part, use Fact \ref{fact:measurecpt} for each $\mathsf{section}^1_x(F^s)$.
	\end{proof}
	
	\subsection{Solovay's measure uniformization and $\GP(\all)$}\label{sec:solovay}
	
	Theorem 9.5 of \cite{Goto2025} shows Solovay's model satisfies $\GP(\all)$.
	In this section we simplify that proof by proving that Solovay's measure uniformization implies $\GP(\all)$.
	
	\begin{defi}
		The \emph{measure uniformization} holds if for every family $\seq{B_x : x \in \R}$ with $\emptyset \ne B_x \subseteq \R$, there is a Lebesgue measurable function $f \colon \R \to \R$ such that $\{ x : f(x) \not \in B_x \}$ is null.
	\end{defi}
	
	Note that the measure uniformization implies every set of reals is Lebesgue measurable.
	
	$\mathsf{CC}_\R$ denotes the countable choice axiom for the reals.
	
	\begin{thm}[$\ZF+\mathsf{CC}_\R$]
		If the measure uniformization holds then $\GP(\all)$ holds.
	\end{thm}
	\begin{proof}
		Let $\seq{A_x : x \in \omega^\omega}$ be a monotone and each $A_x$ is null.
		Suppose $\bigcup_{x \in \omega^\omega} A_x$ is not null.
		Then there is a compact positive set $K \subseteq \bigcup_{x \in \omega^\omega} A_x$.
		Apply the measure uniformization for $\{ (y, x) \in K \times \omega^\omega : y \in A_x \}$.
		Shrinking $K$ and using Lusin's theorem (see, for example, Theorem 17.12 of \cite{kechris2012classical}), we can take continuous $f \colon K \to \omega^\omega$ such that $y \in A_{f(y)}$ for every $y \in K$.
		Since $f[K]$ is compact, we can take an upper bound $x$ of $f[K]$.
		Then we have $K \subset A_x$, which is a contradiction.
	\end{proof}

	\subsection{New necessary condition for $\GP(\all)$}\label{sec:covmeager}
	
	In Proposition 4.4 of \cite{Goto2025}, it was proved that $\GP(\all)$ implies $\add(\meager) \ne \cof(\meager)$.
	In this section, we prove a stronger result.
	
	\begin{thm}
		$\GP(\all)$ implies $\cov(\meager) \ne \frakb$.
	\end{thm}
	\begin{proof}
		Assume $\cov(\meager) = \frakb$ and call this cardinal $\kappa$.
		Then we have $\add(\meager) = \add(\mathcal{E}) = \kappa$.
		We know $\add(\mathcal{E},\nul)=\cov(\meager)$ by Theorem 2.6.14 of \cite{bartoszynski1995set}, where $\add(\mathcal{E},\nul)$ is defined to be the least cardinal $\kappa$ such that there are $\kappa$ many sets from $\mathcal{E}$ whose union is not in $\scrN$.
		So we can take increasing sequence $\seq{A_\alpha : \alpha < \kappa}$ of elements in $\mathcal{E}$ such that $\bigcup_{\alpha < \kappa} A_\alpha \not \in \nul$.
		This is an increasing sequence of null sets of length $\frakb$.
	\end{proof}
	
	\subsection{New sufficient condition for $\GP(\boldpi^1_2)$}\label{sec:gppi12}
	
	In this section, we show that $\boldsig^1_2(\mathbb{B})$ implies $\GP(\boldpi^1_2)$. Let $\mu$ be the Lebesgue measure.
	
	\begin{lem}\label{lem:kechris}
		Let $B$ be a $\boldpi^1_1$ set.
		Then there is $\boldsig^1_2$ formula $\phi(x)$ satisfying:
		\begin{enumerate}
			\item $\phi(x)$ implies $\mu(\proj(B_x))=1$, always, and
			\item Conversely $\mu(\proj(B_x))=1$ implies $\phi(x)$, provided $\boldsig^1_2(\mathbb{B})$.
		\end{enumerate}
	\end{lem}
	\begin{proof}
		This proof is based on Corollary 2.2.2 of \cite{kechris1973337}.
		
		Let us define two formulas $\psi$ and $\phi$ as follows:
		\begin{align*}
			\psi(x, r) &\equiv (\exists B' \text{ Borel set})(B' \subset B_x \land \mu(\proj(B')) > r), \\
			\phi(x) &\equiv (\forall n)\psi(x, 1 - 2^{-n}).
		\end{align*}
		Since $``\mu(\proj(B')) > r"$ part is $\boldsig^1_1$ by Theorem 2.2.3 of  \cite{kechris1973337}, we have $\psi(x, r)$ is $\boldsig^1_2$. Thus $\varphi(x)$ is also $\boldsig^1_2$.
		
		That $\phi(x)$ implies $\mu(\proj(B_x))=1$ is easy to see.
		
		We assume $\mu(\proj(B_x))=1$ and $\boldsig^1_2(\mathbb{B})$, we will deduce $\phi(x)$.
		
		Since $B_x$ is $\boldpi^1_1$, we can uniformize it by $B^*$ in $\boldpi^1_1$.
		We define a function $f$ by
		$f(a) = b \iff (a, b) \in B^*$.
		This map $f$ is $\boldsig^1_2$-measurable.
		
		Thus $f$ is Lebesgue measurable since we assumed $\boldsig^1_2(\mathbb{B})$.
		Therefore we can find a Borel function $g$ with Borel domain $A \subset \proj(B_x)$ such that $\mu(\proj(B_x) \setminus A) = 0$ and $f \upharpoonright A = g$. Put
		$$
		(a, b) \in B' \iff a \in A \land g(a) = b.
		$$
		Then $B'$ is Borel and $\mu(\proj(B')) = 1$.
	\end{proof}
	
	\begin{thm}\label{thm:gppi12}
		$\boldsig^1_2(\mathbb{B})$ implies $\GP(\boldpi^1_2)$.
	\end{thm}
	\begin{proof}
		Let $A \subset \omega^\omega \times 2^\omega$ be a monotone  $\boldpi^1_2$ set with null vertical sections.
		Take a $\boldpi^1_1$ set $B$ by $A_x = 2^\omega \setminus \proj(B_x)$.
		Fix a formula $\phi$ from Lemma \ref{lem:kechris} applied to $B$.
		Since $\mu(\proj(B_x))=1$ for every $x$, we have $\forall x\ \phi(x)$.
		The formula $\forall x\ \phi(x)$ is $\boldpi^1_3$.
		So by $\boldsig^1_2(\mathbb{L})$ which follows from $\boldsig^1_2(\mathbb{B})$, the Laver extension $V[d]$ also satisfies $\forall x\ \phi(x)$.
		It follows that $\mu(\proj(B_x))=1$ for every $x$ in $V[d]$.
		Therefore $\mu(A_d) = 0$ in $V[d]$.
		Since $d$ is a dominating real over $V$, we have $\mu((\bigcup_x A_x)^V)=0$.
		Therefore, by the fact that Laver forcing preserves the Lebesgue outer measure, we have $\mu(\bigcup_x A_x)=0$ in $V$.
	\end{proof}
	A similar argument shows the following.
	
	\begin{thm}\label{thm:gpdelta12}
		$\bolddelta^1_2(\mathbb{B}) \land \boldsig^1_2(\mathbb{L})$ implies $\GP(\bolddelta^1_2)$.
	\end{thm}
	\begin{proof}
		Let $A \in \bolddelta^1_2$.
		Under the assumption $\bolddelta^1_2(\mathbb{B})$, we have the following equivalence.
		$$
		\mu(A_x) > 0 \iff \exists B \text{ Borel set}\ [B \subset A_x \text{ and } \mu(B) > 0]
		$$
		It is $\boldsig^1_3$. So the statement $\mu(A_x) = 0$ is $\boldpi^1_3$. The remaining argument is the same as Theorem \ref{thm:gppi12}.
	\end{proof}
	
	\subsection{$\GP(\all)$ implies that strong measure zero ideal equals null additive ideal}\label{sec:nulladditive}
	
	Since both $\GP(\all)$ and the Borel conjecture hold in the Laver model, it is natural to wonder whether there is an implication between these two.
	This section provides a partial result in the direction that $\GP(\all)$ implies the Borel conjecture.
	
	\begin{lem}
		For every $N \in \nul$, there is a monotone family $\seq{E_J : J  \in \mathsf{IP}}$ of members in $\mathcal{E}$ such that $N \subset \bigcup_{J \in \mathsf{IP}} E_J$. Here, monotonicity means $J \sqsubset J'$ (the standard order of $\mathsf{IP}$) implies $E_J \subset E_{J'}$.
	\end{lem}
	\begin{proof}
		Take a sequence $\seq{C_n : n \in \omega}$ of clopen sets such that $\mu(C_n) \le 2^{-n}$ and $N \subset \{ x : \exists^\infty n\ x \in C_n \}$.
		Put $E_J = \{ x :  \forall^\infty m\ \exists n \in J_m\ x \in C_n \}$.
		Then $\seq{E_J : J  \in \mathsf{IP}}$ is as desired.
	\end{proof}
	
	\begin{thm}
		$\GP(\all)$ implies $\mathcal{SN} = \mathcal{NA}$.
	\end{thm}
	\begin{proof}
		Take $X \in \mathcal{SN}$ and $N \in \nul$.
		We have to show $X + N \in \nul$.
		By the lemma, we take a monotone family $\seq{E_J : J  \in \mathsf{IP}}$ of members in $\mathcal{E}$ such that $N \subset \bigcup_{J \in \mathsf{IP}} E_J$.
		By the characterization of $\mathcal{SN}$, we have $X + E_J \in \nul$ for every $J$.
		Then by $\GP(\all)$, we have $\bigcup_{J \in \mathsf{IP}} (X + E_J) \in \nul$.
		Note that the interval partition version of $\GP(\all)$ is equivalent to the original one since $\omega^\omega$ and $\mathsf{IP}$ are Tukey equivalent.
		Thus, $X + N \in \nul$.
	\end{proof}
	
	\subsection{$\GP(\mathsf{closed}, \mathcal{E})$ fails}
	
	Let $\mathsf{closed}$ denote the pointclass of closed sets.
	We prove $\GP(\mathsf{closed}, \mathcal{E})$ fails in this section. This solves Problem 11.6 of \cite{Goto2025}.
	
	\begin{thm}
		$\GP(\mathsf{closed}, \mathcal{E})$ does not hold.
	\end{thm}
	\begin{proof}
		Let $I$ be the interval partition such that $\abs{I_n} = n + 1$ for every $n$.
		For $x \in 2^\omega$ and $M < \omega$, let
		$$
		C_{y,M} = \abs{\{n \le M : y \restriction I_n \equiv 0 \}}.
		$$
		Also, for each $x \in \omega^{\uparrow \omega}$, let
		$$
		A_x = \{ y \in 2^\omega : \text{for every $k$, $C_{y,f(y)(k)} \ge k+1 $}  \}.
		$$
		Then $A$ is closed and monotone.
		Let
		$$
		B = \{ y \in 2^\omega : \exists^\infty n\ x \restriction I_n \equiv 0 \}.
		$$
		Then $B$ is null and $\bigcup_{x\in \omega^{\uparrow \omega}} A_x \subset B$.
		Thus, each $A_x$ is closed null, so $A_x \in \mathcal{E}$.
		On the other hand, $B$ is comeager, so $B \not \in \mathcal{E}$.
	\end{proof}

	\section{Discussions}
	
	Regarding the contents in Section \ref{sec:hausdorff}, we have the following problems.
	
	\begin{prob}
		Is an axiom of large cardinal necessary for Corollary \ref{cor:gppi11haus}?
	\end{prob}
	
	\begin{prob}
		Is $\GP(\all, \scrN^f_X)$ consistent for every doubling Polish space $X$ and doubling gauge function $f$?
	\end{prob}
	
	Regarding the contents in Section \ref{sec:miscellaneous}, by Theorem \ref{thm:gppi12} and Fact \ref{fact:priorstudies} (\ref{item:regularity}), we have the following table. Each row indicates the countable support iteration of length $\omega_2$.
	
	\begin{table}[H]
		\newcommand{\NO}{\cellcolor{black!40!white}{NO}}
		\newcommand{\YES}{\cellcolor{black!20!white}{YES}}
		\begin{tabular}{l||lllllll}
			& $\boldsig^1_2(\mathbb{B})$ & $\bolddelta^1_2(\mathbb{B})$ & $\boldsig^1_2(\mathbb{L})$ & $\GP(\all)$ & $\GP(\boldsig^1_2)$ & $\GP(\boldpi^1_2)$ & $\GP(\bolddelta^1_2)$ \\
			\hline \hline
			Cohen & \NO & \NO & \NO & \NO & \NO & \NO & \NO \\
			random & \NO & \YES & \NO & \NO & \NO & \NO & \NO \\
			amoeba & \YES & \YES & \YES & \NO & \YES & \YES & \YES \\
			Laver & \NO & \NO & \YES & \YES & \YES & \YES & \YES \\
			Hechler & \NO & \NO & \YES & \NO & ? & ? & ? \\
			Mathias & \NO & \NO & \YES & \NO & ? & ? & ? \\
			Laver*random & \NO & \YES & \YES & \YES & \YES & \YES & \YES \\
			Hechler*random & \NO & \YES & \YES & \NO & ? & ? & \YES \\
			Mathias*random & \NO & \YES & \YES & \NO & ? & ? & \YES
		\end{tabular}
	\end{table}
	
	Here, note that the countable support iteration of Laver*random of length $\omega_2$ also forces $\GP(\all)$.
	The proof is same as Theorem 5.3 of \cite{Goto2025}.

	\begin{prob}
		What are the answers for the question marks in the above table?
	\end{prob}
	
	Finally, regarding the contents in Section \ref{sec:suslinoperator}, we ask:
	
	\begin{prob}
		Does $\GP(\mathcal{A}({\sim} \mathcal{A}(\boldpi^1_1)))$ hold?
	\end{prob}

	\section*{Acknowledgments}
	
	The author would like to thank Jörg Brendle,  Takehiko Gappo, Martin Goldstern, Takayuki Kihara and Diego Mejía, who gave him advice on this research.
	
	\printbibliography[title={References}]
\end{document}